\documentclass[11pt]{article}
\usepackage{amsthm}
\usepackage{fullpage}
\usepackage{amssymb, amsmath}
\usepackage{hyperref}
\usepackage{graphicx}
\usepackage{titlesec}
\usepackage{enumitem}
\usepackage{comment}
\usepackage{relsize}
\newtheorem{thm}{Theorem}

\newtheorem{lem}[thm]{Lemma}
\newtheorem{cor}[thm]{Corollary}

\newtheorem{defi}[thm]{Definition}
\newtheorem{quest}[thm]{Question}
\newtheorem*{claim*}{Claim}
\newcommand{\explode}{\divideontimes}
\newcommand{\cP}{{\mathcal{P}}}

\newcommand{\cI}{{\mathcal{I}}}
\newcommand{\cC}{{\mathcal{C}}}

\title{Degree criteria and stability for independent transversals}
\author{Penny Haxell\thanks{Department of Combinatorics and
    Optimization, University of Waterloo, Waterloo ON
    Canada. Partially supported by NSERC}, Ronen
  Wdowinski\thanks{Department of Combinatorics and 
    Optimization, University of Waterloo, Waterloo ON
    Canada.}}
\date{\today}

\begin{document}

\maketitle

\begin{abstract} 
An \emph{independent transversal} (IT) in a graph $G$ with a given vertex
partition $\cP$ is an independent set of vertices of $G$ (i.e. it
induces no edges), that consists of one vertex from each part
(\emph{block}) of $\cP$. Over the years, various 
criteria have been established that guarantee the
existence of an IT, often given in terms of $\cP$ being
$t$-\emph{thick}, meaning all blocks have size at least $t$. One such
result, obtained recently by Wanless and Wood, is based on the \emph{maximum
  average block degree} 
$b(G,\cP)=\max\{\sum_{u\in U}d(u)/|U|:U\in\cP\}$. They proved
that if $b(G,\cP)\leq t/4$ then an IT exists. Resolving
a problem posed by  Groenland, Kaiser, Treffers and Wales (who showed
that the ratio $1/4$ is best possible), here we give a
full characterization of pairs $(\alpha,\beta)$ such that the
following holds for every $t>0$: whenever $G$ is a 
graph with maximum degree
 $\Delta(G)\leq\alpha t$, and $\cP$ is a $t$-thick vertex partition
of $G$ such that $b(G,\cP)\leq\beta t$, there exists an IT of $G$ with
respect to $\cP$. Our proof makes use of another
previously known criterion
for the existence of IT's that involves the topological connectedness
of the independence complex of graphs, and establishes a general
technical theorem on the structure of graphs for which this parameter
is bounded above by a known quantity. Our result interpolates
between the criterion $b(G,\cP)\leq t/4$ and the old and frequently 
applied theorem that if $\Delta(G)\leq t/2$ then an IT exists. Using
the same approach, we 
also extend a theorem of Aharoni, Holzman, Howard and Spr\"ussel, by giving a
stability version of the latter result.
\end{abstract}

\section{Introduction}

Given a graph $G$ and partition $\mathcal{P} = \{U_1, \ldots, U_r\}$
of its vertex set $V(G)$ into \textit{blocks} $U_i$, an
\textit{independent transversal} (IT) of 
$G$ with respect to $\mathcal{P}$ is an independent set $\{u_1,
\ldots, u_r\}$ of $G$ such that $u_i \in U_i$ for all $i \in [r]$. In
the literature, an independent transversal has also been called an
\textit{independent system of representatives} or a \textit{rainbow
  independent set}.

Many important notions in mathematics can be
described in terms of a suitably chosen graph and vertex partition
having an IT (see e.g.~\cite{Ha3,Habcc} and the references therein),
and accordingly there has been 
much work over the years in proving sufficient
conditions for a vertex-partitioned graph $G$ to have an IT. Many of these
criteria require the block sizes to be large enough with respect to
certain parameters of $G$, and the techniques used to prove such
results have included
purely combinatorial arguments (e.g.~\cite{BeHaSz,Ha1,Ha2,HaSz}),
topological methods 
(e.g.~\cite{AhBe1,AhBeZi,AhChKo,AhHa,Me1,Me2}), arguments using the
Lov\'asz Local Lemma  
and its variants (e.g.~\cite{Al,AlSp,KaKe,LoSu}), and counting arguments
(e.g.~\cite{WaWo}). When no special information 
is known about the structure of the graph with respect to the vertex
partition, the combinatorial and topological methods appear to work
best, giving best possible results in many cases. One example
is the following~\cite{Ha1,Ha2}, where a partition $\cP$ is said to be
$t$-\textit{thick} if each of its blocks has size at least $t$, and as
usual $\Delta(G)$ denotes the maximum degree of $G$. (This is one
of the most frequently applied IT theorems, see
e.g.~\cite{GrHa} and the references therein.)

\begin{thm} \label{max-degree-IT}
Let $G$ be a graph with a $t$-thick vertex partition $\cP$.
If $\Delta(G)\leq t/2$ then $G$ has an IT with
respect to $\mathcal{P}$. 
\end{thm}
Szab\'o and
Tardos~\cite{SzTa} (see also~\cite{Ji,Yu}) proved that
Theorem~\ref{max-degree-IT} is best possible, by giving, for each $d$, a
$(2d-1)$-thick partition $\cP$ of the union of $2d-1$ disjoint copies of the
complete bipartite graph $K_{d,d}$ that does not have an IT with
respect to $\cP$.

When information about the interaction between the graph and the
vertex partition $\cP$ is known,  
in particular when there is some limit on the number of edges between any
pair of blocks or the number of edges incident to any particular block,
then the other techniques mentioned tend to give stronger results. One
recent example of this
phenomenon, due to Wanless and Wood~\cite{WaWo} (see also Kang and
Kelly~\cite{KaKe}), considers the IT 
problem with respect to the \emph{maximum average block degree}
$b(G,\cP)=\max\{\sum_{u\in U}d(u)/|U|:U\in\cP\}$.
\begin{thm} \label{WWthm}
Let $G$ be a graph with a $t$-thick vertex partition $\cP$.
If $b(G,\cP)\leq t/4$ then $G$ has an IT with
respect to $\mathcal{P}$. 
\end{thm}
Kang and Kelly noted how this result can be derived using the Local Cut Lemma of
Bernshteyn~\cite{Be} (see also~\cite{DvEsKaOz}), whereas Wanless and Wood
used a counting argument
similar to that of Rosenfeld~\cite{Ro}. In fact Wanless and Wood
showed that the number of IT's given
by Theorem~\ref{WWthm} is at
least $(t/2)^{|\cP|}$.
\paragraph{Degree considerations.}
In view of Theorem~\ref{max-degree-IT}, it is 
natural to ask whether the $t/4$ in Theorem~\ref{WWthm} can be
improved to $t/2$, and indeed this is stated as an open problem in
both~\cite{KaKe} and~\cite{WaWo}. 
It was resolved by the following theorem of Groenland,
Kaiser, Treffers and Wales~\cite{GrKaTrWa}.
\begin{thm} \label{GKTWthm}
For every $\epsilon>0$ and all sufficiently large $t$, there exists a
forest $F$ and a $t$-thick 
partition $\cP$ of $F$ such that $b(F,\cP)\leq(1+\epsilon)t/4$ and $F$
has no IT with respect to $\cP$.
\end{thm}

(We remark that adding the further
assumption that 
no two blocks of $\cP$ induce a subgraph of $G$ of maximum degree
more than $o(t)$ allows the $t/4$ in Theorem~\ref{WWthm} to be
strengthened all the way to $t-o(t)$ (Glock-Sudakov~\cite{GlSu},
Kang-Kelly~\cite{KaKe}).)

In the main construction giving Theorem~\ref{GKTWthm}, the maximum 
degree $\Delta(F)$ is $t$. In~\cite{GrKaTrWa} the authors speculate
whether it can be 
brought down substantially, and in this direction they
provide a modified construction that is no longer a forest, but has
maximum degree $\alpha t$ with $\alpha$ asymptotically
$\frac12+\frac1{2\sqrt2}<0.854$. Again with 
reference to 
Theorem~\ref{max-degree-IT}, they ask (Problem 12 in~\cite{GrKaTrWa}) 
whether a construction with maximum degree arbitrarily close to $t/2$
is possible. More generally they asked the following.

\begin{quest}\label{GQ}
  What can be said about the set of pairs $(\alpha,\beta)\in[0,1]^2$ such
that for sufficiently large $t$, there exists a graph $G$ with
$\Delta(G)\leq\alpha t$ and a
$t$-thick partition $\cP$ such that $b(G,\cP)\leq\beta t$ and no IT of
$G$ with respect to $\cP$ exists?
\end{quest}
One of the main aims of this paper is to give a complete answer to
Question~\ref{GQ}. To describe it we introduce notion of a
\textit{good pair}.

\begin{defi}\label{defgoodpair}
A pair $(\alpha,\beta)$ with $0<\beta\leq\alpha$ is
a \emph{good pair} if the following holds for every $t>0$: whenever $G$ is a
graph with 
 $\Delta(G)\leq\alpha t$, and $\cP$ is a $t$-thick vertex partition
of $G$ such that $b(G,\cP)\leq\beta t$, there exists an IT of $G$ with
respect to $\cP$.
\end{defi}
Our first main theorem gives a characterisation of the set of good pairs.

\begin{thm}\label{goodpair}
  The pair $(\alpha,\beta)$ with
  $0<\beta\leq\alpha$ 
  is good if and only if one of the following holds.
  \begin{enumerate}
  \item $\alpha\leq 1/2$,
  \item $\beta\leq 1/4$,
    \item $\beta\leq2\alpha(1-\alpha)$.
  \end{enumerate}
\end{thm}
In particular, Theorem~\ref{goodpair} confirms that the value
$\frac12+\frac1{2\sqrt2}$ obtained in~\cite{GrKaTrWa}
is best possible in their setting, i.e. when $\beta$ is asymptotically
$1/4$.

Theorem~\ref{goodpair} can be viewed as an
interpolation between Theorem~\ref{max-degree-IT}
and Theorem~\ref{WWthm}, refining both results when $\alpha>1/2$ and
$\beta>1/4$. Our proof of Theorem~\ref{goodpair}
uses the topological approach of~\cite{AhBe1,AhHa,Me1,Me2} (see
also~\cite{Habcc}), relating the existence of IT's 
in $G$ to the parameter $\eta(G)$, which is defined as the
topological (homotopic) connectedness of the independence complex of
$G$, plus $2$ (see Section~\ref{sec-conn}). This approach is based on
the following topological Hall or Rado theorem, first
proven explicitly in \cite{Me2} and generalized in \cite{AhBe1} (with
variants appearing earlier in \cite{AhHa} and \cite{Me1}). Here for a
subset $S$ of blocks of the partition $\cP=\{U_1, \ldots, U_r\}$ of
$V(G)$, we write $G_S= G \left[ \bigcup_{U_i \in S} U_i
  \right]$ for the subgraph of $G$ induced by the union of the
blocks in $S$.

\begin{thm}\label{exists-IT}
Let $G$ be a graph with vertex partition $\cP$. 
If $\eta(G_S)\geq|S|$ for every subset $S$ of the blocks of $\cP$,
then $G$ has an IT with respect to $\mathcal{P}$. 
\end{thm}
We derive (one implication of) Theorem~\ref{goodpair} by first proving
a general technical theorem 
(Theorem~\ref{key} in Section~\ref{sec-key}) about the structure of graphs $H$
for which an upper bound on $\eta(H)$ is known. Theorem~\ref{key} has
the following as a simple consequence.

\begin{thm}\label{dense}
Let $H$ be a graph with $n$ vertices and maximum degree $d$.
Suppose that $\eta(H)\leq k$. Then
$|E(H)|\geq dn-d^2k$.
\end{thm}
Since an upper bound on $b(G,\cP)$ easily implies an upper bound
on the density $|E(G_S)|/|V(G_S)|$ of any $G_S$, from here it is
a short step to combine 
Theorem~\ref{dense} for $H=G_S$ with Theorem~\ref{exists-IT} to prove
this implication of Theorem~\ref{goodpair}. (See Section~\ref{sec-dense}.)

We provide constructions to prove the other implication of
Theorem~\ref{goodpair} in Section~\ref{sec-construct}.
\begin{thm} \label{tight-construction}
For every $\alpha > \frac{1}{2}$, $\beta > \max \left\{ \frac{1}{4},
2\alpha(1 - \alpha) \right\}$, and $t$ sufficiently large, there
exists a graph $G$ and a $t$-thick partition $\cP$ of $V(G)$ with
$\Delta(G)\leq\alpha t$ and $b(G,\cP)\leq\beta t$, such that
$G$ has no IT with respect to $\cP$.
\end{thm}
Our main tool for the proof of Theorem~\ref{tight-construction} is a simple
general lemma (Lemma~\ref{join-lemma}) that 
gives a way of constructing a graph and partition with no IT from two
smaller ones. Lemma~\ref{join-lemma} is quite versatile and is used
in the Appendix as well, and also discussed further in
Section~\ref{sec-conclusion}.  

\paragraph{Stability.}
Another feature of the topological approach to proving the existence
of IT's via Theorem~\ref{exists-IT} is that it can also give
information about extremal or near-extremal configurations.  
A straightforward example of this is the (topological) proof of
Theorem~\ref{max-degree-IT}, which amounts to combining Theorem~\ref{exists-IT}
with the fact~\cite{Me1} that every graph $G$ with maximum degree $d$
and at least $2dk$ vertices satisfies $\eta(G) \ge k$. As a warm-up
in using the topological technique, we give a proof of this fact in
Section~\ref{sec-conn}, 
together with a characterization of the extremal examples, which turn
out to be disjoint unions of $K_{d,d}$'s.

Going further in this direction, 
Aharoni, Holzman, Howard, and Spr\"ussel~\cite{AhHoHoSp} showed by
these means that
for $d=\Delta(G)\geq 3$, if $\cP$ in
Theorem~\ref{max-degree-IT} is $(2d-1)$-thick and $G$ has no IT then
$G$ contains the union of 
$2d-1$ disjoint $K_{d,d}$'s.
The second main aim of this paper is to prove a \textit{stability} version of
this, which states that if $G$ does not have an
IT and $\cP$ is $t$-thick where $t$ is close to $2d$, then for some
subset $S$ of blocks of $\cP$, the graph $G_S$ is
close to being the union of disjoint $K_{d,d}$'s. Again we will derive
our result by combining Theorem~\ref{exists-IT} with another
consequence of our main technical result Theorem~\ref{key}, which is
as follows.

\begin{thm}\label{witheps}
Let $G$ be a graph with $n$ vertices and maximum degree $d$, and let
$\epsilon$ be such that 
$n(1+\epsilon)=2dk$. Suppose that $\eta(G) \leq k$. Then $\epsilon
\geq 0$ and $2|E(G)|\geq dn(1-\epsilon)$, and $V(G)$
has a partition into sets $\{X_i,Y_i,Z_i\}_{i=1}^k$ such that
$\sum_{i=1}^k|Z_i| \leq \epsilon n$ and the following hold: 
\begin{itemize}
\item[(i)] $\left|E(G)\setminus\bigcup_{i=1}^kE(X_i,Y_i)\right| \leq
  2\epsilon dn$,  
\item[(ii)] $\sum_{i=1}^k|X_i||Y_i|-\left|\bigcup_{i=1}^kE(X_i,Y_i)\right|
  \leq \epsilon dn/2$. 
\end{itemize}
\end{thm}
Thus when $|V(G)|$ is close to $2dk$ but $\eta(G) \leq k$, there are
two respects in which $G$ is close in structure to the disjoint union
$J$ of complete bipartite graphs with vertex sets
$\{(X_i,Y_i)\}_{i=1}^k$: at most a 
small proportion of the edges of $G$ are not edges of $J$, and almost
all the edges of $J$ are edges of $G$. Let us say that the graph $G$
is $\gamma$-\emph{approximated by} $J$ if $V(J)\subseteq V(G)$ and the
symmetric difference 
of $E(G)$ and $E(J)$ has size at most $\gamma|E(G)|$.  
Combining Theorem~\ref{exists-IT} with Theorem~\ref{witheps} applied
to $G_S$ leads to the following
stability version of Theorem~\ref{max-degree-IT}.

\begin{thm} \label{stab-IT}
Let $0<\beta<1$ and let $G$ be a graph with a $t$-thick vertex
partition $\cP$. 
If $d:=\Delta(G)\leq (1+\beta)t/2$ and $G$ has no IT with
respect to $\mathcal{P}$, then for some subset $S$ of blocks of $\cP$,
the graph $G_S$ satisfies $|E(G_S)|>(1-\beta)d|V(G_S)|/2$, and $G_S$ is
$5\beta/(1-\beta)$-approximated by a disjoint  
union $J$ of $|S|-1$ complete bipartite graphs.
\end{thm}
The simple derivation of Theorem~\ref{stab-IT}
appears in Section~\ref{sec-stability}. Another
stability version of Theorem~\ref{max-degree-IT}, that 
concludes
that $G$ as in Theorem~\ref{stab-IT} contains a large disjoint union
of complete bipartite graphs, is discussed in
Section~\ref{sec-conclusion}.

In the next section we introduce the
topological method, and describe the tools that we will use in this
paper. Section~\ref{sec-key} is devoted to the proof of our main
technical result, Theorem~\ref{key}, and a companion result that will
be useful in the subsequent sections. Our main results,
Theorem~\ref{goodpair} and Theorem~\ref{witheps}, are proved in
Sections~\ref{sec-dense} and~\ref{sec-stability} respectively. Our
constructive result, Theorem~\ref{tight-construction}, is proved in
Section~\ref{sec-construct}. We end the paper with some concluding remarks
in Section~\ref{sec-conclusion}.

\section{Topological connectedness and independent transversals}\label{sec-conn}

The set  $\cI(G)$ of all independent sets in a graph $G$ forms an abstract
simplicial complex (i.e. a ``closed-down'' set), called the
\emph{independence complex} of $G$. 
An abstract simplicial complex $\cC$ is said to be $k$-{\it connected} if
for each $-1\leq d\leq k$ and each continuous map $f$ from the sphere
$S^d$ to $||\cC||$ (the body of the geometric realization of $\cC$),
the map $f$ can be extended to a continuous map from the ball
$B^{d+1}$ to $||\cC||$. The \emph{connectedness} of $\cC$ is the
largest $k$ for which $\cC$ is $k$-connected. Following~\cite{AhBe1},
we define the graph 
parameter $\eta(G)$ to be 2 plus the connectedness of $\cI(G)$. The
link between topology and the existence of IT's was 
first discovered in~\cite{AhHa}, and developed much more fully in
subsequent works such as~\cite{AhBe1,Me1,Me2}. (See
e.g. the survey~\cite{Habcc} for an in-depth discussion of these
notions, and some intuition on why they relate to IT's.)

Theorem \ref{exists-IT} implies that we can obtain sufficient
conditions for independent transversals from lower bounds on
connectedness. Conversely, constructing graphs with no independent
transversals is often aided by finding graphs $G$ with a low value of
$\eta(G)$. Understanding when this parameter
is small will be our main motivation for
the upcoming Theorem~\ref{key}, which will form the basis of our work
in this paper.

As observed in~\cite{AhBe1}, adding 2 to the connectedness of $\cI(G)$
in the definition of $\eta(G)$
simplifies various statements about this parameter, such as the
following basic facts.

\begin{lem} \label{connected-lemma}
\begin{enumerate}[label={(\alph*)}]
    \item Every graph $G$ has $\eta(G) \ge 0$, and $\eta(G) = 0$ if
      and only if $G$ is empty. 
    \item If graph $G$ has an isolated vertex, then $\eta(G) = \infty$.
    \item If $G$ and $H$ are disjoint graphs, then $\eta(G \cup H) \ge
      \eta(G) + \eta(H)$. 
\end{enumerate}
\end{lem}

For a graph $G$ and edge $e \in E(G)$, we denote by $G - e$ the graph
obtained by deleting $e$. 
We write $G \explode e$ for the graph formed by \textit{exploding}
$e$, which is obtained from $G$ by removing
both endpoints of $e$ together with all of their neighbours. In
other words, if $x$ and $y$ are the endpoints of $e$, then $G \explode
e$ is the subgraph induced by $V(G)\setminus (N(x) \cup N(y))$. Our
main tool for obtaining lower bounds 
for the $\eta$ parameter will be the following theorem of
Meshulam~\cite{Me2} (see also e.g.~\cite{AhBeZi}).

\begin{thm}[Meshulam] \label{Meshulam}
If $G$ is a graph and $e \in E(G)$ is an edge, then
\begin{align*}
    \eta(G) \ge \min\{\eta(G - e), \eta(G \explode e) + 1\}.
\end{align*}
\end{thm}

An edge $e$ of a graph $G$ is called \textit{deletable} if $\eta(G -
e) \le \eta(G)$. An edge $e$ is called \textit{explodable} if $\eta(G
\explode e) \le \eta(G) - 1$. By Theorem~\ref{Meshulam}, every edge of
$G$ is either deletable or explodable. A graph $G$ is called
\textit{reduced} if no edge is deletable (hence every edge is
explodable). A subgraph $G' \subseteq G$ is called a
\textit{reduction} of $G$ if $G'$ is reduced, $V(G') = V(G)$, and
$\eta(G') \le \eta(G)$. One may obtain a reduction of a graph $G$ by
iteratively deleting deletable edges until there are none left. 

The topological proof of Theorem~\ref{max-degree-IT} is based on the
lower bound $\eta(G) \ge 
\frac{|V(G)|}{2\Delta(G)}$ from~\cite{Me1}. As an introduction
to working with the 
connectedness parameter $\eta$, we give a proof of this bound, along with a
characterization of the graphs $G$ for which the bound is tight. 

\begin{thm}\label{extremal-graphs}
Let $G$ be a graph with maximum degree $d$. If $|V(G)| \ge 2dk$, then
$\eta(G) \ge k$. Further, if $\eta(G) = k$ then $G$ is the disjoint
union of $k$ copies of $K_{d,d}$.  
\end{thm}

\begin{proof}
We use induction on $k$. By
Lemma~\ref{connected-lemma}(1), the result is immediate for $k = 0$,
and for $k =  
1$, the hypothesis $|V(G)| > 0$ implies that
$\eta(G) \ge 1$. Suppose that $|V(G)| \ge 2d$ and $\eta(G) =
1$. Having $\eta(G) < 2$ means that the independence complex of $G$ is
not path-connected. This is equivalent to saying that the complement
of $G$ is disconnected, since a topological path between two vertices
in the independence complex of $G$ corresponds to a graph-theoretical
path between those vertices in the complement of $G$. This then
implies that $G$ has a complete cut, that is, $G$ contains a complete
bipartite graph on all of its vertices. Since $|V(G)| \ge 2d$ and
$\Delta(G) = d$, the
only possibility is that $G$ is the complete 
bipartite graph $K_{d,d}$. This proves the base case $k = 1$. 

Assume that $k \ge 2$ and the conclusion of the theorem holds for
all values smaller than $k$. By iteratively deleting deletable edges from 
$G$, we may assume that $G$ is a reduced graph, so that every edge of
$G$ is explodable. Let $x$ be a vertex of minimum degree in $G$. Then
$d(x)>0$ by Lemma~\ref{connected-lemma}(2), so we may choose
an edge $e$ of $G$ incident to $x$. Let
$G' = G \explode e$, and let $d'$ denote the
maximum degree of $G'$. Then $d'\leq d$, and 
we may calculate that $|V(G')| \ge
2dk - (d(x)+d) \geq 2d(k - 1)\geq2d'(k-1)$. Hence by
the induction  
hypothesis we conclude $\eta(G') \ge k - 1$, and if $\eta(G') =
k - 1$ then $G'$ is the union of $k - 1$ disjoint copies of
$K_{d',d'}$. Since $xy$ is explodable in $G$, it follows that
$\eta(G) \ge \eta(G')+1\geq k$ as required.

For the second assertion, suppose that $\eta(G)=k$. Then equality
holds everywhere in the previous paragraph, so that $d'=d$ and $G'$ is
the union of $k - 1$ disjoint copies of $K_{d,d}$.
Since $\Delta(G) = d$
and $G'$ is $d$-regular, the graph $H = G[N(x) \cup N(y)]$ is a
connected component of $G$, and $G = G' \cup H$. This implies that
$\eta(H) \le \eta(G) - \eta(G') = 1$ and $|V(H)| = |V(G)| - |V(G')|
\ge 2d$. From the base case $k = 1$, it follows that $H$ is a
$K_{d,d}$. Therefore, $G$ is the union of $k$ disjoint copies of
$K_{d,d}$, finishing the induction. 
\end{proof}
Observe that Theorem~\ref{extremal-graphs} then leads to (a slightly
stronger version of) Theorem~\ref{max-degree-IT}. 

\begin{cor} \label{max-degree-IT-corollary}
Let $G$ be a graph with  $\Delta(G)=d$, and let $\mathcal{P} = \{U_1,
\ldots, U_m\}$ be a partition of $V(G)$. If  
\begin{align*}
	\sum_{i \in S} |U_i| \ge 2d(|S| - 1) + 1 \hspace{0.5cm} \text{ for every } S \subseteq \cP,
\end{align*}
then $G$ has an IT with respect to $\cP$. In particular, if $\cP$ is
$t$-thick and $d\leq t/2$ then $G$ has an IT.
\end{cor}

\begin{proof}
By Theorem \ref{extremal-graphs}, for every $S \subseteq \cP$ we have
\begin{align*}
	\eta(G_S) \ge \frac{1}{2d} \left|\mathsmaller{\bigcup}_{i \in S} U_i \right| \ge \frac{2d(|S| - 1) + 1}{2d} > |S| - 1,
\end{align*}
so that $\eta(G_S) \ge |S|$. By Theorem \ref{exists-IT}, $G$ has an IT.
\end{proof}

\section{Graphs with low connectedness}\label{sec-key}
In this section we describe our main technical result on the structure
of graphs whose independence 
complexes have low topological connectedness. It is
helpful to keep in mind that the partition in Theorem~\ref{key} is not
related to any given partition into blocks as we have been discussing
in the previous sections, but is rather a different partition that
comes about as a consequence of the upper bound on $\eta(G)$.

\begin{thm} \label{key} 
Let $G$ be a graph with $\eta(G)< \ell$. Then $V(G)$ has a partition
into sets $\{X_i,Y_i,Z_i\}_{i=1}^k$ for some $k < \ell$, such that  
\begin{itemize}
\item[(a)] for each $i$, there exist $x_i\in X_i$ and $y_i\in Y_i$
  with $Y_i\cup Z_i\subseteq N(x_i)$ and $X_i\cup Z_i\subseteq
  N(y_i)$, 
\item[(b)] $4\sum_{i=1}^k |X_i||Y_i|+\sum_{i=1}^k (|X_i|+|Y_i|)|Z_i|
  \leq 2|E(G)| + 2\sum_{i=1}^k |E(X_i,Y_i)|$. 
\end{itemize}
\end{thm}

\begin{proof}
We construct a sequence of subgraphs $G_0, H_0, G_1, H_1, \ldots, G_k,
H_k$ with $G_0 = G$ and $H_k$ having no edges, such that $H_i$ is a
reduction of $G_i$ and $G_i$ is obtained from $H_{i-1}$ by exploding
an edge $x_iy_i$ that kills the smallest number of vertices. Note that
every edge of each $H_i$ is explodable, and that $k$ explosions are
performed in total. Since $\eta(G) < \ell$, Theorem \ref{Meshulam}
implies that $k < \ell$. 

For $1 \le i \le k$, let $Z_i = N(x_i) \cap N(y_i)$, $X_i = N(y_i)
\setminus Z_i$, and $Y_i = N(x_i) \setminus Z_i$. Note that
$\{X_i,Y_i,Z_i\}_{i=1}^k$ is a partition of $V(G)$, that $x_i \in X_i$
and $y_i \in Y_i$, and that $Y_i\cup Z_i\subseteq N(x_i)$ and $X_i\cup
Z_i\subseteq N(y_i)$. We classify the edges of $G$ as follows. If an
edge joins two vertices in distinct sets among $\{X_i, Y_i, Z_i\}$ for
some $i$, then we label it as \textit{black}. All other edges are
labeled as \textit{pink}. Further, we assign a direction to each pink
edge that joins some $X_i \cup Y_i \cup Z_i$ to $X_j \cup Y_j \cup
Z_j$ with $i < j$ by orienting it froms its $i$ end to its $j$ end. 

Fix $1 \le i \le k$. Consider any vertex $w \in Y_i$ ($\subseteq
N(x_i)$). By the choice of the edge $x_iy_i$, the vertex $w$ has at
least $|X_i|$ neighbours outside of $N(x_i)$ in $H_{i-1}$. The edges
joining $w$ to $X_i$ are black. Denoting their number by $b_{X_i}(w)$,
this implies that at least $|X_i| - b_{X_i}(w)$ edges incident to $w$
in $H_{i-1}$ are pink out-edges. Similarly, for every vertex $w \in
X_i$, denoting the number of edges joining to $w$ to $Y_i$ by
$b_{Y_i}(w)$, at least $|Y_i| - b_{Y_i}(w)$ edges incident to $w$ in
$H_{i-1}$ are pink out-edges. Now consider any vertex $w \in Z_i$ ($=
N(x_i) \cap N(y_i)$). Let $b_{X_i}(w)$ and $b_{Y_i}(w)$ denote the
black degree of $w$ into $X_i$ and $Y_i$, respectively. Again by the
choice of the edge $x_iy_i$, the vertex $w$ has at least $|X_i|$
neighbours outside of $N(x_i)$ and at least $|Y_i|$ neighbours outside
of $N(y_i)$. This tells us that $w$ has pink outdegree at least $|X_i|
- b_{X_i}(w)$ and at least $|Y_i| - b_{Y_i}(w)$, and so in particular
at least $(|X_i| + |Y_i| - b_{X_i}(w) - b_{Y_i}(w))/2$. 

We get that the total pink outdegree of vertices in $X_i \cup Y_i \cup
Z_i$ is at least 
\begin{align*}
    \sum_{w \in X_i} (|Y_i| - b_{Y_i}(w)) + \sum_{w \in Y_i} (|X_i| -
    b_{X_i}(w)) + \sum_{w \in Z_i} (|X_i| + |Y_i| - b_{X_i}(w) -
    b_{Y_i}(w))/2. 
\end{align*}
Note that $\sum_{w \in X_i} b_{Y_i}(w) = \sum_{w \in Y_i} b_{X_i}(w) =
|E(X_i, Y_i)|$ and that $\sum_{w \in Z_i} b_{X_i}(w) + b_{Y_i}(w) =
|E(X_i \cup Y_i, Z_i)|$. Therefore, summing over $1 \le i \le k$, the
total pink outdegree of vertices of $G$ is at least 
\begin{align*}
    \sum_{i = 1}^k \left( 2|X_i||Y_i| - 2|E(X_i, Y_i)| + (|X_i| +
    |Y_i|)|Z_i|/2 - |E(X_i \cup Y_i, Z_i)|/2 \right). 
\end{align*}

Next we give an upper bound on the total pink indegree of vertices of
$G$. As before, each vertex $w \in Y_i$ has at least $|X_i|$ neighbours
outside of $N(x_i)$ in $H_{i-1}$, and thus it has at most $d(w) -
|X_i|$ neighbours in $\bigcup_{j<i} X_j\cup Y_j\cup Z_j$, since these
vertices are not present in $H_{i-1}$. Likewise, each vertex $w \in
X_i$ has at most $d(w) - |Y_i|$ neighbours in $\bigcup_{j<i}X_j\cup
Y_j\cup Z_j$. Finally, each vertex $w \in Z_i$ has at least
$(|X_i|+|Y_i|-b_{X_i}(w)-b_{Y_i}(w))/2$ pink out-neighbours, plus
$b_{X_i}(w)+b_{Y_i}(w)$ neighbours in $X_i \cup Y_i$. Thus its pink
indegree is at most
$d(w)-(|X_i|+|Y_i|+b_{X_i}(w)+b_{Y_i}(w))/2$. Therefore the total pink
indegree is at most 
\begin{align*}
    \sum_{i=1}^k \left( \sum_{w \in Y_i} (d(w) - |X_i|) + \sum_{w \in
      X_i} (d(w) - |Y_i|) + \sum_{w \in Z_i} (d(w) - \frac{|X_i| +
      |Y_i| + b_{X_i}(w) + b_{Y_i}(w)}{2} \right) \\ 
    = 2|E(G)| - \sum_{i=1}^k 2|X_i||Y_i| - \sum_{i=1}^k (|X_i| +
    |Y_i|)|Z_i|/2 - \sum_{i=1}^k |E(X_i \cup Y_i, Z_i)|/2. 
\end{align*}

Combining the lower bound on the total pink outdegree and upper bound
on the total pink indegree, we find that 
\begin{align*}
&\sum_{i = 1}^k \left( 2|X_i||Y_i| - 2|E(X_i, Y_i)| + (|X_i| +
  |Y_i|)|Z_i|/2 - |E(X_i \cup Y_i, Z_i)|/2 \right) \\ 
&\leq 2|E(G)| - \sum_{i=1}^k 2|X_i||Y_i| - \sum_{i=1}^k (|X_i| +
  |Y_i|)|Z_i|/2 - \sum_{i=1}^k |E(X_i \cup Y_i, Z_i)|/2. 
\end{align*}
Then some rearranging gives that
\begin{align*}
    4\sum_{i=1}^k |X_i||Y_i|+\sum_{i=1}^k (|X_i|+|Y_i|)|Z_i| \leq
    2|E(G)| + 2\sum_{i=1}^k |E(X_i,Y_i)|, 
\end{align*}
as required.
\end{proof}

We call a partition $\{X_i,Y_i,Z_i\}_{i=1}^k$ of $V(G)$ a \emph{basic
  partition} if it satisfies the conclusions of Theorem~\ref{key}. We  
next establish another technical result about basic partitions, for use
in the next two 
sections.

\begin{lem}\label{techlem}
Let $G$ be a graph with $n$ vertices with maximum degree $\Delta(G) = d$.
Suppose that $\{X_i,Y_i,Z_i\}_{i=1}^k$ is a
basic partition of $V(G)$. Then
$$2|E(G)|\geq 2Q+2dn-2d^2k+\sum_{i=1}^k|Z_i|(2d-|X_i|-|Y_i|-2|Z_i|),$$
where $Q=\sum_{i=1}^k |X_i||Y_i|-\sum_{i=1}^k |E(X_i,Y_i)|$.
\end{lem}

\begin{proof}
For each $i$, we set $r_i = |X_i|+|Y_i|$ and assume without loss
of generality that $|X_i|\geq|Y_i|$. 

Observe that, for any positive integers $x, r$ with $r-x\leq x<r$, we
have that $x(r-x)>(x+1)(r-x-1)$. Hence with these conditions the
expression $x(r-x)$ is smallest when $x$ is as large as possible. By
Property (a) of the basic partition we know that $|X_i|\leq d-|Z_i|$,
and this implies that  
$$|X_i||Y_i|\geq(d-|Z_i|)(r_i+|Z_i|-d)=dr_i+2d|Z_i|-d^2-|Z_i|r_i-|Z_i|^2.$$

Recalling $Q=\sum_{i=1}^k |X_i||Y_i|-\sum_{i=1}^k |E(X_i,Y_i)|$ we
obtain from (b) that
\begin{align*}
  2|E(G)| & \geq 2Q+2\sum_{i=1}^k |X_i||Y_i|+\sum_{i=1}^k (|X_i|+|Y_i|)|Z_i|\\
 & \geq 2Q+ 2\sum_{i=1}^k(dr_i+2d|Z_i|-d^2-|Z_i|r_i-|Z_i|^2)+\sum_{i=1}^k|Z_i|r_i\\
&=2Q+2d\sum_{i=1}^k(r_i+|Z_i|)-2d^2k+2d\sum_{i=1}^k|Z_i|-\sum_{i=1}^k|Z_i|r_i-2\sum_{i=1}^k|Z_i|^2\\
  &=2Q+ 2dn-2d^2k+\sum_{i=1}^k|Z_i|(2d-r_i-2|Z_i|),
\end{align*}
where in the last line we use the fact that $\{X_i,Y_i,Z_i\}_{i=1}^k$
is a partition of $V(G)$. This finishes the proof.
\end{proof}

\section{Density and the proof of Theorem~\ref{goodpair}}\label{sec-dense}

We begin by noting that Theorem~\ref{key} and Lemma~\ref{techlem}
immediately imply Theorem~\ref{dense}.

\begin{proof}[Proof of Theorem~\ref{dense}] Let
$H$ with $\Delta(H)=d$ and $\eta(H)\leq k$ be given. Let $s\leq k$ be
  such that $\{X_i,Y_i,Z_i\}_{i=1}^s$ is a 
basic partition of $V(H)$, which exists by Theorem~\ref{key}. We apply
  Lemma~\ref{techlem}, first noting that clearly $Q\geq 
  0$. By Property (a) of the basic partition we find that
  $|X_i|+|Y_i|+2|Z_i|\leq2d$, 
  and hence the terms $(2d-|X_i|-|Y_i|-2|Z_i|)$ in the conclusion of
  Lemma~\ref{techlem} are all non-negative. Hence $2|E(G)|\geq
  2dn-2d^2s\geq 2dn-2d^2k$, thus completing the proof.
\end{proof}



We may now give the proof of Theorem~\ref{goodpair} (assuming the
result of Theorem~\ref{tight-construction}).

\begin{proof}[Proof of Theorem~\ref{goodpair}]
If (1) holds ($\alpha\leq1/2$), then for every $t$, and every $G$ and
$\cP$ as in Definition~\ref{defgoodpair}, we see that $G$ and $\cP$
satisfy the conditions of Theorem~\ref{max-degree-IT}. Hence $G$ has
an IT with respect to $\cP$. If (2) holds ($\beta\leq1/4$), then every
$t$, $G$ and 
$\cP$ satisfy the assumptions of Theorem~\ref{WWthm}, hence again there
exists an IT of $G$ with respect to $\cP$.

Suppose (3) holds ($\beta\leq2\alpha(1-\alpha)$). By the previous
paragraph we may assume 
$\alpha>1/2$ and $\beta>1/4$, and hence also clearly $\alpha<1$. Let $t$,
$G$ and $\cP$ be as in  
Definition~\ref{defgoodpair}, so that in particular
$\Delta(G)\leq\alpha t$, and suppose on the contrary that $G$ has
no IT with respect to $\cP$. Then by Theorem~\ref{exists-IT}
there exists a 
subset $S$ of blocks of $\cP$ such that the subgraph $G_S$
of $G$ induced by $\bigcup_{U\in S}U$ satisfies
$\eta(G_S)\leq|S|-1$. Let us define $\gamma$ by $\gamma
t=\Delta(G_S)$. Then clearly $\gamma\leq \alpha$. Also, since $|V(G_S)|\geq
t|S|$, we see that $\eta(G_S)\geq \frac{t|S|}{2\gamma t}$ by
Theorem~\ref{extremal-graphs}, from which we conclude that
$\gamma\geq\frac{|S|}{2(|S|-1)}>\frac12$. 

Observe that for the set $S$ (or indeed any set) of blocks of $\cP$ we
have $2|E(G_S)|\leq b(G,\cP)\sum_{W\in 
  S}|W|\leq\beta t|V(G_S)|$. Recall that each block of
$\cP$ has size at least $t$. Hence, noting that $2\gamma-\beta>0$ since 
$\beta\leq \alpha< 1$ and $\gamma>1/2$, and applying Theorem~\ref{dense}
to $G_S$, we get
\begin{align*}
  \beta t|V(G_S)| & \geq2|E(G_S)|\geq2\gamma t|V(G_S)|-2(\gamma t)^2(|S|-1)\\
  (2\gamma-\beta)t|V(G_S)| & \leq2(\gamma t)^2(|S|-1)\\
  (2\gamma-\beta)t^2|S| & \leq2\gamma^2t^2(|S|-1)\\
  (2\gamma-\beta) & <2\gamma^2 \\
  \beta & > 2\gamma(1-\gamma).
\end{align*}
Since $\gamma\leq\alpha<1$ and $\gamma>1/2$, we observe
$2\gamma(1-\gamma)\geq 2\alpha(1-\alpha)$. Hence we find a  
contradiction to the assumption $\beta\leq 2\alpha(1-\alpha)$. This
concludes the proof that $(\alpha,\beta)$ is good. 

Finally suppose that none of (1--3) hold. Then by
Theorem~\ref{tight-construction} there exist $t$, $G$ and $\cP$  as in 
Definition~\ref{defgoodpair} such that $G$ has no IT with respect to
$\cP$, showing that $(\alpha,\beta)$ is not a good pair.
\end{proof}

\section{Stability and the proof of
  Theorem~\ref{witheps}}\label{sec-stability}

\begin{proof}[Proof of Theorem~\ref{witheps}]
Let $G$ be a graph with $\eta(G)\leq k$, and set $n=|V(G)|$,
$d=\Delta(G)$, and define $\epsilon$ by $n(1+\epsilon)=2dk$. Then
$\epsilon\geq 0$ by  Theorem~\ref{extremal-graphs}.
Theorem~\ref{dense} tells us that $|E(G)|\geq dn-d^2k=dn-d(n+\epsilon
n)/2$ and hence $2|E(G)|\geq dn-\epsilon dn$ as claimed in
Theorem~\ref{witheps}.

By
Theorem~\ref{key} there exists a basic partition
$\{X_i,Y_i,Z_i\}_{i=1}^s$ of $V(G)$ for some 
$s\leq k$. If $s<k$ then define $ X_i=Y_i=Z_i=\emptyset$ for $s+1\leq
i\leq k$.
Set $z:=\sum_{i=1}^s|Z_i|=\sum_{i=1}^k|Z_i|$. Using Property (a) of our basic partition we get
$$z=\sum_{i=1}^s(|X_i|+|Y_i|+2|Z_i|)-\sum_{i=1}^s(|X_i|+|Y_i|+|Z_i|)\leq2ds-n\leq
2dk-n.$$
Hence $z\leq\epsilon n$ since $n(1+\epsilon)=2dk$.

Our next aim is to establish Conclusion (i) of Theorem~\ref{witheps}.
We assume without loss 
of generality that $|X_i|\geq|Y_i|$. Setting $r_i=|X_i|+|Y_i|$ and
proceeding exactly as in the proof of Lemma~\ref{techlem}, we may
conclude for each $i\leq s$ that
$$|X_i||Y_i|\geq(d-|Z_i|)(r_i+|Z_i|-d)=dr_i+2d|Z_i|-d^2-|Z_i|r_i-|Z_i|^2.$$
Setting $m = |E(G)|-\sum_{i=1}^s|E(X_i,Y_i)|$, we find from Property
(b) of the basic partition that
\begin{align*}
4|E(G)|-2m & \geq
4\sum_{i=1}^s(dr_i+2d|Z_i|-d^2-|Z_i|r_i-|Z_i|^2)+\sum_{i=1}^s|Z_i|r_i\\ 
&=4d\sum_{i=1}^s(r_i+|Z_i|)-4d^2s+4d\sum_{i=1}^s|Z_i|-3\sum_{i=1}^s|Z_i|r_i-4\sum_{i=1}^s|Z_i|^2\\  
&= 4dn-4d^2s+\sum_{i=1}^s|Z_i|(4d-3r_i-4|Z_i|),
\end{align*}
where in the last line we again use the fact that $\{X_i,Y_i,Z_i\}_{i=1}^s$
is a partition of $V(G)$. 

By Property (a) we know that $r_i\leq 2d-2|Z_i|$ for each $i$, from
which we get that 
\begin{align*}
4|E(G)|-2m &\geq 4dn-4d^2s + \sum_{i=1}^s|Z_i|(4d-6d+2|Z_i|)\\
&= 4dn-4d^2s+\sum_{i=1}^s|Z_i|(2|Z_i|-2d)\\
&= 4dn-4d^2s+2\sum_{i=1}^s|Z_i|^2-2d\sum_{i=1}^s|Z_i|\\
&\geq 4dn-4d^2s+\frac2s(\sum_{i=1}^s|Z_i|)^2-2d\sum_{i=1}^s|Z_i|\\
&= 4dn-4d^2s+\frac{2z^2}s-2dz.
\end{align*}
Therefore
$$m\leq 2|E(G)|-dn -dn+2d^2s-\frac{z^2}s+dz.$$
Since $d=\Delta(G)$ we know $2|E(G)|\leq dn$, so using also the facts
that $s\leq k$ and $z\leq\epsilon
n$, and that $n(1+\epsilon)=2dk$, we conclude
$$|E(G)\setminus\bigcup_{i=1}^s E(X_i,Y_i)|=m\leq 2d^2s-dn+dz\leq
2d^2k-dn+\epsilon dn=2\epsilon dn,$$
which is the statement of Conclusion (i) since $X_i=Y_i=\emptyset$ for $s+1\leq i\leq k$.

For Conclusion (ii), we know by
Lemma~\ref{techlem} that
$$2|E(G)|\geq 2Q+2dn-2d^2s+\sum_{i=1}^s|Z_i|(2d-|X_i|-|Y_i|-2|Z_i|),$$
where $Q=\sum_{i=1}^s |X_i||Y_i|-\sum_{i=1}^s |E(X_i,Y_i)|$.
As before, the quantity $(2d-|X_i|-|Y_i|-2|Z_i|)$ is
nonnegative by Property (a) of basic partitions, and so
$$2Q\leq2|E(G)|-2dn+2d^2s=(2|E(G)|-dn)+(2d^2s-dn)\leq 0+2d^2k-dn=\epsilon dn,$$
which is the statement of (ii).
\end{proof}

We end this section with the proof of Theorem~\ref{stab-IT}.

\begin{proof}[Proof of Theorem~\ref{stab-IT}]
Let $0<\beta<1$ be given, and let $G$ be a graph with a $t$-thick vertex
partition $\cP$, for which
$d=\Delta(G)\leq (1+\beta)t/2$. Suppose $G$ has no IT with
respect to $\mathcal{P}$. Then by Theorem~\ref{exists-IT}, there exists a
subset $S$ of blocks of $\cP$ for which $\eta(G_S)\leq k$, where
$k=|S|-1$. Set $n=|V(G_S)|$ and define
$\epsilon$ by $n(1+\epsilon)=2dk$. Since $\cP$ is a
$t$-thick partition we know $n\geq t|S|$, and so it follows that
      $$\epsilon=\frac{2dk}{n}-1\leq
\frac{2(1+\beta)t(|S|-1)}{2t|S|}-1<1+\beta -1=\beta.$$ 
By Theorem~\ref{witheps} we know $|E(G_S)|\geq
(1-\epsilon)dn/2>(1-\beta)dn/2$. 

Let 
  $\{X_i,Y_i,Z_i\}_{i=1}^k$ be the partition of $V(G_S)$ guaranteed
  by Theorem~\ref{witheps}, and let $J$ be the union of the $k$ complete
    bipartite graphs induced by  $\{X_i,Y_i\}_{i=1}^k$. Then by
      Theorem~\ref{witheps}, the symmetric difference of $E(G_S)$ and
      $J$ has size at most
      $$2\epsilon dn+\epsilon dn/2=5\epsilon dn/2\leq
      5\frac{\epsilon}{1-\epsilon}|E(G_S)|.$$ 
      Hence $G_S$ is $5\epsilon/(1-\epsilon)$-approximated by
      $J$. Noting that $5\epsilon/(1-\epsilon)< 5\beta/(1-\beta)$
      completes the proof.
\end{proof}

\section{Constructing graphs with no independent
  transversals}\label{sec-construct} 

This section is devoted to the proof of Theorem~\ref{tight-construction}.
Let us call a pair $(\alpha,\beta)$ \emph{relevant} if it satisfies
the conditions of Theorem~\ref{tight-construction}, in other words 
$\alpha > \frac{1}{2}$ and $\beta > \max \left\{ \frac{1}{4},
2\alpha(1 - \alpha) \right\}$. Our aim is to construct, for every
relevant pair $(\alpha,\beta)$ and every
sufficiently large $t$, a graph $G$ and a partition $\cP =
\{U_1, \ldots, U_r\}$ of $V(G)$ such that
\begin{itemize}
	\item[(a)] for each $i$ we have $|U_i|\geq t$;
	\item[(b)] $G$ has maximum degree $\Delta(G)\leq\alpha t$; 
	\item[(c)] $b(G,\cP)\leq\beta t$; and
          \item[(d)] $G$ has no IT with respect to $\cP$.
\end{itemize}
As mentioned in the Introduction, in ~\cite{GrKaTrWa}
Groenland, Kaiser, Treffers and Wales modified their proof of
Theorem~\ref{GKTWthm} to provide also a proof of
Theorem~\ref{tight-construction} for every $\epsilon>0$ and every
relevant pair $(\alpha,\frac14+\epsilon)$. As $\epsilon$ approaches
zero, this gives the asymptotic value 
$\frac{1}{2} + \frac{1}{2\sqrt{2}}<0.854$ for $\alpha$. A
slight generalization of their construction, discussed in
Section~\ref{sec-conclusion}, 
is sufficient to establish Theorem~\ref{tight-construction} for
$\alpha \ge \frac{2}{3}$. However, modifications are needed to get a
construction that works for all $\alpha > \frac{1}{2}$.

Our main tool is the following simple lemma, that allows us to build
complicated graphs with no IT starting from simpler ones. 
\begin{lem} \label{join-lemma}
Let $J$ and $H$ be disjoint graphs, and let $\mathcal{P} = \{U_1,
\ldots, U_r\}$ and $\mathcal{Q} = \{W_1, \ldots, W_s\}$ be partitions
of $V(J)$ and $V(H)$ respectively, such that $J$ has no IT with
respect to $\mathcal{P}$ and $H$ has no IT with respect to
$\mathcal{Q}$. Let $\mathcal{R} = \{U_1', \ldots, U_r', W_1, \ldots,
W_{s-1}\}$, where $U_1' \supseteq U_1, \ldots, U_r' \supseteq U_r$ are
obtained by distributing each of the vertices in $W_s$ into one of
$U_1, \ldots, U_r$ arbitrarily. Then $J \cup H$ has no IT with respect
to $\mathcal{R}$. 
\end{lem}

\begin{proof}
Assume for contradiction that $J \cup H$ has an IT $\{u_1, \ldots,
u_r, w_1, \ldots, w_{s-1}\}$ with respect to $\mathcal{R}$, where $u_i \in
U_i'$ for $1 \le i \le r$ and $w_i \in W_i$ for $1 \le i \le s-1$. If
$u_i \in U_i$ for every $1 \le i \le r$, then $\{u_1, \ldots, u_r\}$
is an IT of $J$ with respect to $\mathcal{P}$, a contradiction. So
suppose instead that $u_j \in U_j' \cap W_s$ for some $1 \le j \le
r$. Then $\{w_1, \ldots, w_{s-1}, u_j\}$ is an IT of $H$ with respect
to $\mathcal{Q}$, again a contradiction. 
\end{proof}

In the proof that follows, for a given relevant pair
$(\alpha,\beta)$ and (sufficiently large) $t$, we
will be applying Lemma \ref{join-lemma} possibly many 
times. In each application,
the complete bipartite graph $K=K_{t-\lfloor\alpha
  t\rfloor+1,\lfloor\alpha t\rfloor}$ with its standard bipartition
into two blocks will be used in place of $J$. Clearly $K$ has no IT
with respect to this partition. We will refer to the
block $A$ of $K$ as the 
``$A$-side" and the block $B$ as the ``$B$-side", where $t-\lfloor\alpha
  t\rfloor+1=|A|<|B|=\lfloor\alpha t\rfloor$. 

Our proof of Theorem~\ref{tight-construction} builds directly upon the
construction from~\cite{GrKaTrWa} that established their special case of
Theorem~\ref{tight-construction} (i.e. when $\beta$ is close to
$1/4$). Given the pair 
$(\alpha,\beta)$ with $\alpha\geq\beta>1/4$, this construction provides,
for all sufficiently large $t$, a graph $G'$ and  
partition $\cP'$ with Properties (b), (c) and (d), and with Property
(a) replaced by
\begin{itemize}
	\item[(a$'$)] Every block $U$ with $|U|<t$ consists of
          vertices of degree 1 and satisfies $|U|=\lfloor\alpha t\rfloor$.
	\end{itemize}
For completeness we describe $(G',\cP')$ explicitly (formulated using
Lemma~\ref{join-lemma}) in the Appendix.

\begin{proof}[Proof of Theorem~\ref{tight-construction}]
  Let $(\alpha,\beta)$ be a relevant pair, and let $t$ be large
  enough such that $(G',\cP')$ satisfying Properties (a$'$), (b), (c)
  and (d) exists (as given in~\cite{GrKaTrWa}). We
  call blocks of 
  size less than $t$ \emph{deficient}. If the set of 
deficient blocks in $(G',\cP')$ is empty then 
$(G',\cP')$ itself proves
Theorem~\ref{tight-construction} for $(\alpha,\beta)$, so we may  
assume the contrary from now on. In particular by (a$'$) we may assume that
$\alpha<1$.

Our plan is to create a new vertex-partitioned graph by adding a
number of copies of  $K=K_{t-\lfloor\alpha   t\rfloor+1,\lfloor\alpha
  t\rfloor}$ to $G'$ and suitably extending the current partition.
Throughout this process, the current graph $G$ and partition 
$\cP$ will be \emph{valid}, meaning that $(G,\cP)$ satisfies Properties (b),
(c) and (d) together with 
\begin{itemize}
	\item[(a$''$)] Every deficient block $U$ satisfies
          $|U|\geq\lfloor\alpha t\rfloor$, and $d(u)\leq
          t-\lfloor\alpha t\rfloor+1$ for each $u\in U$. 
	\end{itemize}
Observe that the initial pair $(G',\cP')$ is valid.
\begin{figure*}
  \begin{center}
    \leavevmode
    \includegraphics[scale=0.4]{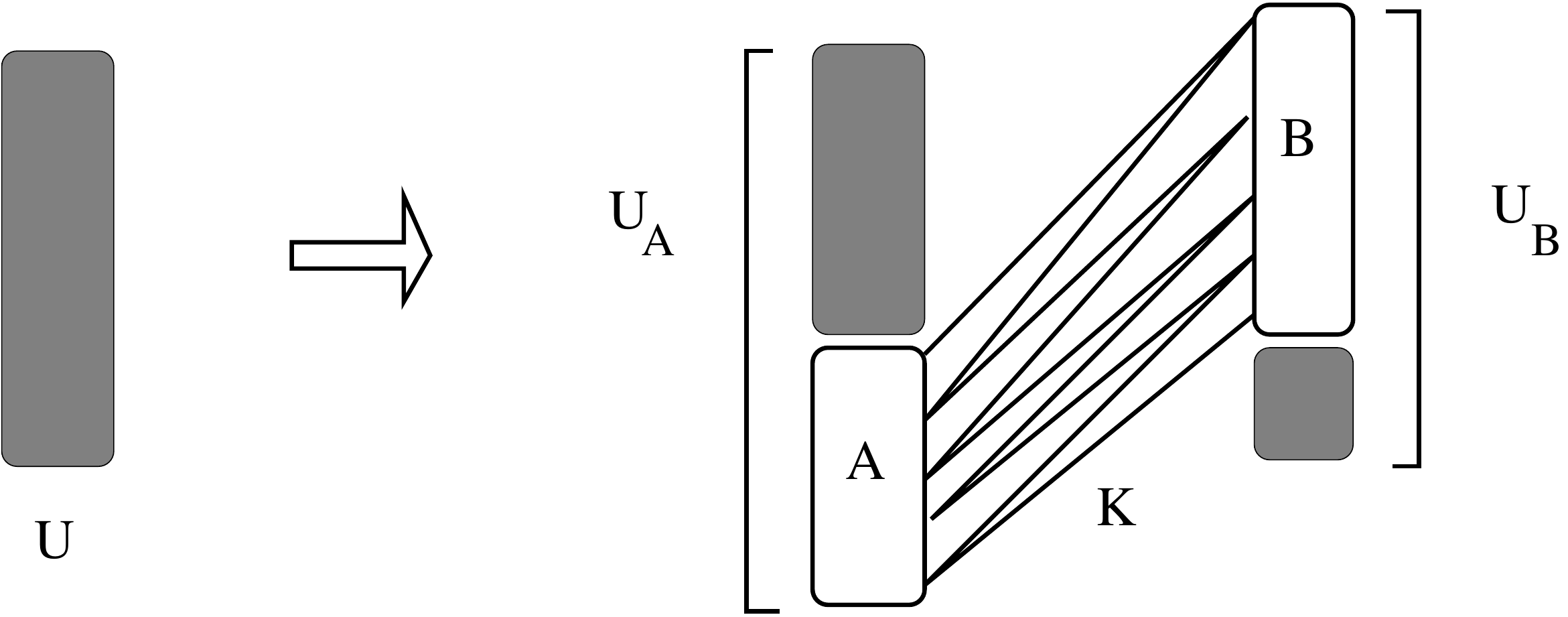}
  \end{center}
  \caption{The deficient block $U$ is replaced
  with blocks $U_A$ of size $t$ and $U_B$ of size $|U|+1$.}
\end{figure*}

Our main construction step is as follows (see Figure 1). Given a valid pair
$(G,\cP)$, if no blocks are deficient then $(G,\cP)$ provides a proof
of Theorem~\ref{tight-construction} and we stop. Otherwise, we take a
deficient block $U$, add a disjoint copy of $K$ to $G$, distribute
$\lfloor\alpha t\rfloor-1$ vertices of $U$ to the $A$-side of $K$ to
form a new block $U_A$, and
the remaining $|U|-\lfloor\alpha t\rfloor+1$ vertices of $U$ to the $B$-side
of $K$, forming block $U_B$. Note that this process
removes $U$, adds 
one new block $U_A$
of size $t$ and one new block $U_B$ of size $|U|+1$. All other blocks remain
unchanged. Thus the resulting pair
$(G^+,\cP^+)$ satisfies (a$''$), and Lemma~\ref{join-lemma} tells us
that (d) is also satisfied. Property (b) holds because
$\Delta(K)=\lfloor\alpha t\rfloor$ and $(G,\cP)$ satisfies
(b). To verify (c), note that since 
$(G,\cP)$ satisfied (a$''$), the new
block $U_A$ has average degree at most
\begin{align*}
	\frac{1}{t}((t - \lfloor\alpha t\rfloor + 1) 
        \lfloor\alpha t\rfloor + (\lfloor\alpha t\rfloor - 1)  (t
        - \lfloor\alpha t\rfloor + 1))=\frac{1}{t}(2\lfloor\alpha t\rfloor - 1) 
        (t-\lfloor\alpha t\rfloor + 1)<\frac{2\alpha t}t(t-\alpha t+2).
\end{align*}
This quantity is
$2\alpha(1-\alpha)t+4\alpha$, which for large enough $t$ is at most
$\beta t$ as required. 

Each vertex in the other new block $U_B$ has degree exactly $t -
\lfloor\alpha t\rfloor + 1$ (if it came from $K$) or at most $t -
\lfloor\alpha t\rfloor + 1$ by (a$''$) (if it came from $U$). Since $t
- \lfloor\alpha t\rfloor + 1< t-\alpha t+2<2 \alpha (1 - \alpha)t+2$,
we have verified (c), and 
hence that $(G^+,\cP^+)$ is valid. Therefore we may continue the
construction.

Observe that after each step of this construction, the deficient block
$U$ is replaced with a larger block $U_B$, and all other
deficient blocks are unchanged. Thus we may repeat the construction
step finitely many times to obtain a valid pair $(G,\cP)$ that has no
deficient blocks. This
completes the proof of  Theorem~\ref{tight-construction}.
\end{proof}

\section{Concluding remarks}\label{sec-conclusion}

The constructions given in Section~\ref{sec-construct} and the
Appendix are by no means unique. In the main construction step in
Section~\ref{sec-construct} for example, given a deficient block $U$,
we added $\lfloor\alpha 
t\rfloor-1$ vertices of $U$ to the $A$-side of $K=K_{t-\lfloor\alpha
  t\rfloor+1,\lfloor\alpha t\rfloor}$ and 
the remaining $|U|-\lfloor\alpha t\rfloor+1$ vertices of $U$ to the $B$-side
of $K$. We could instead fix any positive integer $C$ and, in each
step, add $\lfloor\alpha 
t\rfloor-C$ vertices of $U$ to the $A$-side of $K'=K_{t-\lfloor\alpha
  t\rfloor+C,\lfloor\alpha t\rfloor}$ and 
the remaining $|U|-\lfloor\alpha t\rfloor+C$ vertices of $U$ to the $B$-side
of $K'$. This would require our minimum value for $t$ to be
larger, but as a trade-off would give constructions with fewer blocks.
The bipartite component
$K'$ and its number of copies would also
be different. However, we know by Theorem~\ref{stab-IT} that 
when $\alpha$ is close to $\frac{1}{2}$, the constructed graph with no
IT should have many copies of ``near" $K_{t/2, t/2}$'s, so in any case
we would get many components that are close in structure to complete
bipartite graphs.

For the case in which
$\alpha$ is slightly
greater than $\frac{1}{2} + \frac{1}{2\sqrt{2}}$ and $\beta$ slightly
greater than $\frac{1}{4}$, a proof of
Theorem~\ref{tight-construction} is given in~\cite{GrKaTrWa} as
follows. Recall that the basic graph $G'$ with partition $\cP'$ satisfies
Properties (a$'$), (b), (c) and (d) described in
Section~\ref{sec-construct}. For each deficient block $U$ of $\cP'$, a
single complete 
bipartite graph $K'=K_{\lfloor\beta' t/\alpha\rfloor,
  \lfloor\alpha t\rfloor}$ is added, where 
$\beta' = \max \left\{ \frac{1}{4}, 2\alpha(1 - \alpha) \right\}$, and
the vertices of $U$ are distributed into the two blocks of $K'$ so
that the size of each becomes at least $t$. This is possible because
$\left( t - \frac{\beta' t}{\alpha} \right) + (t - \alpha t) \le (1 -
2(1 - \alpha))t + (1 - \alpha)t = \alpha t$, and $|U|=\lfloor\alpha t\rfloor$. As before,
Lemma~\ref{join-lemma} ensures that the new graph has no IT with respect to
the new partition. The maximum block average
degree does indeed stay below $\beta$ after this distribution. However, for
this strategy to work for a wider range of $(\alpha,\beta)$  we would need that $\beta' t / \alpha \le \alpha t$
to ensure $\Delta(K')\leq\alpha t$ (Property (b)). Hence this
proves Theorem~\ref{tight-construction} in general only when
$\alpha \ge \frac{2}{3}$. 

We remark that our simple construction lemma in
Section~\ref{sec-construct}, Lemma~\ref{join-lemma}, provides a
convenient way to describe (and in some cases generalize) various
known constructions for graphs without IT's. For example this is done
in the Appendix to derive the construction of~\cite{GrKaTrWa}. The same
approach can also provide new 
constructions for other contexts. These topics are explored further
in~\cite{HaWd}. 

As mentioned in the Introduction, it is possible to modify
Theorem~\ref{key} to show that if $G$ has a low value of $\eta(G)$ then
it will contain many disjoint complete bipartite subgraphs. The
precise technical statement is as follows.

\begin{thm}\label{key2} Fix $\epsilon>0$ and let $G$ be a graph with
  maximum degree $d$ such that 
  $n=|V(G)|=(2-\epsilon)kd$ and $\eta(G)\leq k$. Let $\gamma>3$ be
  given. Then there exist $k_0\leq k$ and a partition 
  $V(G)=W\cup\bigcup_{i=1}^{k_0}(X_i\cup Y_i\cup Z_i)$ such that the
  following hold. 
\begin{itemize}
\item[(a)] for each $i$ there exist $x_i\in X_i$ and $y_i\in Y_i$ with
  $Y_i\cup Z_i\subseteq N(x_i)$ and $X_i\cup Z_i\subseteq N(y_i)$,
\item[(b)] For $\zeta:=\frac1{dk}\sum_i|Z_i|$ we have
  $\zeta\leq \epsilon$ and 
  $k-k_0\leq\frac{1+\gamma}4(\epsilon-\zeta)k$, and 
$$|W|\leq\gamma(\epsilon-\zeta)dk.$$
\item[(c)] the set $U$ of vertices that are not conforming satisfies
$$|U|\leq\frac{\gamma^2-1}{\gamma^2-3\gamma}|W|+
\frac{1+\gamma}{\gamma-3}\epsilon dk\leq\frac{\gamma^2-1}{\gamma-3}(\epsilon-\zeta)dk+\frac{1+\gamma}{\gamma-3}\epsilon dk,$$
  where $w$ is said to be \emph{conforming} if for some $i$ we have
  $w\in X_i$ and $Y_i\cup Z_i\subseteq N(w)$, or $w\in Y_i$ and
  $X_i\cup Z_i\subseteq N(w)$. 
\end{itemize}
\end{thm}

To interpret this statement it helps to compare it with
Theorem~\ref{key}. In Theorem~\ref{key2} we have an additional part $W$ in the
partition, consisting of ``junk'' over which we have little control,
but by (b) its size is small as long as $\epsilon$ is small. Property
(a) is the same as in 
Theorem~\ref{key}. The subgraph of $G$ induced by the set of
conforming vertices in each $X_i\cup 
Y_i\cup Z_i$ contains a spanning complete bipartite graph, and
Property (c) tells 
us that most vertices are conforming, provided also that $\gamma$ is
not too close to 3. Hence under these conditions $G$ contains $k_0$
complete bipartite graphs, whose union contains most of the vertices
of $G$, and where $k_0$ is close to $k$.

The proof of Theorem~\ref{key2} is not difficult, and is quite similar
to the proof of Theorem~\ref{key}. However, including it
here would add a significant amount of technical detail, so we 
omit the full proof in favour of the following sketch.

\begin{proof}[Sketch of proof]
Since $\eta(G)<k+1$ we know that every sequence of
edge deletions and explosions taking $G$ to the empty graph has length
at most $k$. We define a specific sequence in a step-by-step fashion,
and derive the structural conclusion about $G$ from the fact that it
has at most $k$ explosion steps. 
As before, the
operation of {\it reducing} a graph $G$ is to 
delete edges of $G$ one by one until no further edges are
  deletable. Then in the resulting {\it reduced} graph, every edge is
  explodable.

  Set $\theta=\frac{\gamma -3}{1+\gamma}$. Then
  $\gamma=\frac{3+\theta}{1-\theta}$. 
\begin{enumerate}
\item Reduce $G$. Set $i:=1$.
\item If graph is empty, stop. If not, then it will contain no
  isolated vertices. Choose an edge $xy$ whose explosion
  kills the smallest 
  possible number of vertices. Set $Z:=N(x)\cap N(y)$ in the current
  graph, and set $X:=N(y)\setminus Z$ and $Y:=N(x)\setminus Z$. 


  Call a vertex $w\in X$ (respectively $Y$) {\it high} if it has at
least $\theta d$ 
neighbours outside $N(x)=Y\cup Z$ (respectively $N(y)=X\cup Z$), and
{\it low} otherwise. 

IF there exists a low vertex $w\in X$ (respectively in $Y$) that has a
non-neighbour $u\in Y$ (respectively $u\in X$):
\begin{itemize}
\item explode $xu$ (respectively $yu$) and reduce,
\item explode $wv$ for some neighbour $v$ of $w$ and reduce,
\item put all vertices lost in the two explosions into $W$.
\end{itemize}

  ELSE set $x_i:=x$ and $y_i:=y$. Set $X_i:=X$, $Y_i:=Y$ and $Z_i:=Z$.
    Explode
    $x_iy_i$ and reduce.
\item Increment $i$ and repeat from (2). 
\end{enumerate}
The rest of the proof consists of the analysis of this procedure, and
how it leads to Conclusions (a--c).
\end{proof}

In just the same way that Theorem~\ref{key} (via
Theorem~\ref{witheps}) combined with 
Theorem~\ref{exists-IT} led to the IT stability result
Theorem~\ref{stab-IT}, Theorem~\ref{key2} combined with
Theorem~\ref{exists-IT} leads to an alternative IT stability
result. This one asserts that if a graph $G$ does not have an
IT with respect to a $t$-thick partition $\cP$, where $\Delta(G)$ is close to
$t/2$, then for some
subset $S$ of blocks of $\cP$, the graph $G_S$ contains
a disjoint union of complete bipartite graphs that spans almost all of
$V(G_S)$. 


\section{Appendix}\label{sec-appendix}
Here we describe (using Lemma~\ref{join-lemma}) the
construction given in~\cite{GrKaTrWa} of a graph $G'$ 
with vertex partition $\cP'$ that has the properties (a$'$), (b), (c)
and (d) detailed in Section~\ref{sec-construct}.

Let $(\alpha,\beta)$ be any
pair with $\alpha\geq\beta>1/4$.
We fix $t \ge 1$, and a sequence $d_1 < d_2 < \cdots < d_k = \min\{t,
\lfloor\alpha t\rfloor\}$ of increasing  
positive integers. We will indicate the conditions we will need on $t$
and the sequence as our argument progresses. We inductively define
graphs $G_1,
G_2, \ldots G_k$ and partitions $\mathcal{P}_1, \mathcal{P}_2, \ldots,
\mathcal{P}_k$ such that each $G_i$ has no IT with respect to
$\cP_i$. Each $G_{i+1}$ is obtained from $G_i$ by adding a collection
of components, each of which is a copy of the
star $K_{1,d_{i+1}}$. As in Section~\ref{sec-construct} we refer
  to the singleton set formed by the centre vertex of $K_{1,d_{i+1}}$
    as its $A$-side, and its set of $d_{i+1}$ leaves as its $B$-side.
    We start with $G_1$ being the disjoint 
union of $t$ copies of $K_{1, d_1}$. The partition $\mathcal{P}_1$ has
as blocks each of the $B$-side blocks (each of size $d_1$), plus one
more block of size $t$ (the \emph{initial} block) that is
the union of all the $A$-side blocks. It
is easy to see that $G_1$ has no IT with respect to
$\mathcal{P}_1$. (This fact can also be derived by applying Lemma
\ref{join-lemma} successively to $K_{1,d_1}$.)  

Set $d_0=0$, $G_0=\emptyset$ and $\cP_0=\emptyset$. Suppose $1\leq j<k$ and that
we have defined $G_j$ and $\mathcal{P}_j$ such that the following hold.
\begin{itemize}
\item[(i)] $G_j$ has no IT with respect to $\cP_j$.
\item[(ii)] Any block $U$ of $\cP_j$ with size different from $t$ 
  satisfies $U\cap
  V(G_{j-1})=\emptyset$ and $|U|=d_j$, and each $u\in U$ has degree 1
  in $G_j$. These are called the \emph{terminal} blocks of $\cP_j$.
  \item[(iii)] 
    Any block $U$ of $\cP_j$ with $|U|=t$ 
    that contains a vertex of  
    $V(G_{j})\setminus V(G_{j-1})$ has average degree 
    $\frac{1}{t}(d_{j-1} + (t - d_{j-1})d_{j})$ in $G_j$ with respect to $\cP_j$. 
\end{itemize}
We define $G_{j+1}$ and
$\mathcal{P}_{j+1}$ as follows. We will essentially be adding a new
``layer" of star components $K_{1,d_{j+1}}$. By (ii) we know
that all blocks of $\mathcal{P}_j$ have size $t$ except the terminal
blocks of $\cP_j$, each of which has size
$d_j$. Fix one terminal block $X$, and let $X_1 = X$. With
Lemma~\ref{join-lemma} in mind, for each $i = 1, \ldots, t - d_j$ in
sequence, we add to 
the current graph a copy of $K_{1, d_{j+1}}$, distribute all vertices
of $X_i$ into the $A$-side of $K_{1, d_{j+1}}$ to form a new block
$X_{i+1}$ of size $|X_i| + 1$, and let the $B$-side of $K_{1,
  d_{j+1}}$ form one more new block. 
(This $B$-side of $K_{1, d_{j+1}}$ will become a
terminal block of the partition $\cP_{j+1}$.) For each $i$,
Lemma~\ref{join-lemma} 
guarantees that the new graph has no IT with respect to the 
new partition. Repeating this step until $i$ reaches $t-d_j$, we
obtain $|X_{t - d_j + 1}| = t$, 
so this finishes the process of enlarging $X$ to size $t$. We repeat this
procedure for every terminal block $X$ of $\cP_j$, creating a
new layer of star components $K_{1, d_{j+1}}$, whose $B$-sides are the
terminal blocks of the new partition $\cP_{j+1}$. We denote the new
graph by $G_{j+1}$.

Next we verify (i--iii) for $(G_{j+1},\cP_{j+1})$. 
By repeated applications of Lemma~\ref{join-lemma} as indicated, we
know that $G_{j+1}$ has no IT with respect to
$\mathcal{P}_{j+1}$. Hence (i) holds for $(G_{j+1},
\mathcal{P}_{j+1})$. To check (ii), observe that by (ii) for $(G_{j},
\mathcal{P}_{j})$ and by our construction, the 
only blocks that are not now of size $t$ are the $B$-side blocks of the
newly added layer of copies of  $K_{1, d_{j+1}}$, which all 
have size $d_{j+1}$, and their vertices all have degree 1. Thus $(G_{j+1},
\mathcal{P}_{j+1})$ satisfies (ii). To verify (iii), the blocks $U$ of
$\cP_{j+1}$ with $|U|=t$ that contain a vertex of  
    $V(G_{j+1})\setminus V(G_{j})$ are precisely those that were previously a
terminal block $X$ of $\cP_j$. By Property (ii) of
$(G_j,\cP_j)$ we know that $|X|=d_j$ and each vertex of $X$ has degree
1 in $G_j$. Our construction adds $t-d_j$ vertices to $X$, each of degree
$d_{j+1}$, to obtain the class $U$ of size $t$. Hence the average
degree of $U$ is $\frac{1}{t}(d_j + (t - d_j)d_{j+1})$, verifying
(iii) for $(G_{j+1},\mathcal{P}_{j+1})$.

Let $G' = G_k$ and $\mathcal{P}' = \mathcal{P}_k$. 
Observe that $G'$ has maximum degree $d_k=\min\{t,\lfloor\alpha t\rfloor\}
\le \alpha t$, so Property (b) is satisfied. Property (d) is given by
(i). To ensure that Property (c) is satisfied, we use the following lemma
from~\cite{GrKaTrWa}. 

\begin{lem} \label{sequence-lemma}
For every $\beta > \frac{1}{4}$, there exists $t_0 \ge 1$ such that
for all $t \ge t_0$, there exists a sequence $0=d_0<d_1 < d_2 < \cdots <
d_k = \min\{t,\lfloor\alpha t\rfloor\}$ of integers such that for
all $j \ge 0$, 
\begin{align*}
	\frac{1}{t}(d_j + (t - d_j)d_{j+1}) \le \beta t.
\end{align*}
\end{lem}
For example, we may set $d_1=\lfloor\beta t\rfloor$ and, as noted
in~\cite{GrKaTrWa}, one can simply take $d_j = d_{j-1}+1$ for 
every $j \in \{2, \ldots, d_k-d_1+1\}$. Then we get
\begin{align*}
	\frac{1}{t}(d_j + (t - d_j)d_{j+1}) &= \frac{1}{t}(d_j + (t -d_j
        )(d_j+1))= \frac{1}{t}(t + d_j(t - d_j)).
\end{align*}
Since $d_j\leq d_k= \min\{t,\lfloor\alpha
t\rfloor\}\leq t$, this expression is at most $\frac{1}{t}\left(t +
\frac{t^2}{4}\right)$. This is bounded above by $\beta t$ provided
$t$ is sufficiently large, since $\beta>1/4$. 

We choose $t_0$ and a sequence such that Lemma~\ref{sequence-lemma}
holds for our given $\beta$, and apply our construction with this
sequence and with $t\geq t_0$. Any terminal block of $\cP'=\cP_k$ has
average degree 
1 by (ii) with $j=k$. All remaining blocks in $\cP'$ have
size $t$ and are of the 
type in (iii) for exactly one value of $j$, and hence by (iii) and
Lemma~\ref{sequence-lemma} each has average degree at most $\beta
t$. Hence (c) holds for $(G',\cP')$.

Finally, to check (a$'$) we again note that any block $U$ that has
size less than $t$ is a 
terminal block of $\cP_k$, and hence satisfies (ii). Then necessarily
$U$ has size $d_k=\lfloor\alpha t\rfloor<t$, and by (ii) consists of
vertices of degree one. This verifies (a$'$) as required. 



\end{document}